\documentclass[11pt,reqno]{amsart}

\usepackage[utf8]{inputenc}
\usepackage[a4paper]{geometry}

\RequirePackage{amsmath,amsfonts,amssymb,bbm,mathtools,enumerate,amsthm,mathrsfs,crossreftools,stmaryrd,setspace,color,dsfont,bm,breakurl}
\RequirePackage[inline]{enumitem}
\RequirePackage[authoryear]{natbib}
\RequirePackage[colorlinks=true,linkcolor=blue,citecolor=blue,urlcolor=blue,pdfborder={0 0 0}]{hyperref}

\numberwithin{equation}{section}
\theoremstyle{plain}
\newtheorem{theorem}{Theorem}[section]
\newtheorem{lemma}{Lemma}[section]
\theoremstyle{definition}
\newtheorem{PropQ}{Property}

\makeatletter
\newcommand{\optionaldesc}[2]{%
	\phantomsection
	#1\protected@edef\@currentlabel{#1}\label{#2}%
}
\makeatother

\allowdisplaybreaks[4]

\begin{document}
	
\title[Convergence in distribution of the quantile process]{A necessary and sufficient condition for convergence in distribution of the quantile process in $\bm{L^1(0,1)}$}

\author{Brendan K.\ Beare}
\address{School of Economics,
	University of Sydney, City Road.
	Darlington, NSW 2006, Australia.}
\email{brendan.beare@sydney.edu.au}

\author{Tetsuya Kaji}
\address{Booth School of Business, 
	University of Chicago, 5807 S.\ Woodlawn Ave. 
	Chicago, IL 60637, USA.}
\email{tkaji@chicagobooth.edu}

\date{\today}

\maketitle

\begin{center}
Accepted for publication in \emph{Bernoulli}.
\end{center}

\begin{abstract}
We establish a necessary and sufficient condition for the quantile process based on iid sampling to converge in distribution in $L^1(0,1)$. The condition is that the quantile function is locally absolutely continuous and satisfies a slight strengthening of square integrability. If the quantile process converges in distribution then it may be approximated using the bootstrap.
\end{abstract}

\section{Introduction}\label{sec:intro}

Let $L^1(0,1)$ be the space of Lebesgue integrable functions $f:(0,1)\to\mathbb R$ equipped with the norm $\lVert f\rVert_1\coloneqq\smallint_0^1\lvert f(u)\rvert\,\mathrm{d}u$, with the customary understanding that functions differing only on a null set are regarded to be equal. This article establishes a necessary and sufficient condition for the quantile process constructed from an independent and identically distributed (iid) sample to converge in distribution in $L^1(0,1)$. Let $X_1,\dots,X_n$ be iid random variables with cumulative distribution function (cdf) $F:\mathbb R\to[0,1]$ and quantile function $Q:(0,1)\to\mathbb R$. Let $F_n:\mathbb R\to[0,1]$ and $Q_n:(0,1)\to\mathbb R$ be the empirical cdf and empirical quantile function defined by $F_n(x)=n^{-1}\sum_{i=1}^n\mathbbm 1(X_i\leq x)$ and $Q_n(u)=\inf\{x\in\mathbb R:F_n(x)\geq u\}$. We will show that the quantile process $\sqrt{n}(Q_n-Q)$ converges in distribution in $L^1(0,1)$ if and only if $Q$ has the following property.

\begin{PropQ}\label{prop:Q}
	The function $Q:(0,1)\to\mathbb R$ is locally absolutely continuous and satisfies
	\begin{equation}\label{eq:Qmoment}
		\int_{0}^1\sqrt{u(1-u)}\,\mathrm{d}Q(u)<\infty.
	\end{equation}
\end{PropQ}

Local absolute continuity of $Q$ is defined to mean that the restriction of $Q$ to each compact subinterval of $(0,1)$ is absolutely continuous. The integral in \eqref{eq:Qmoment} is defined in the Lebesgue--Stieltjes sense, wherein $Q$ is used to define a unique Radon measure on $(0,1)$. Requiring the integral in \eqref{eq:Qmoment} to be finite is slightly more restrictive than requiring $Q$ to be square integrable: it is shown in \citet[pp.~64--5]{H73} that if $Q$ satisfies \eqref{eq:Qmoment} then $\smallint_0^1Q(u)^2\,\mathrm{d}u<\infty$, and that if there exists $\delta>0$ such that $\smallint_0^1Q(u)^2[\log(1+\lvert Q(u)\rvert)]^{1+\delta}\,\mathrm{d}u<\infty$ then $Q$ satisfies \eqref{eq:Qmoment}.

Prior literature has established a necessary and sufficient condition for the empirical process $\sqrt{n}(F_n-F)$ to converge in distribution in $L^1(\mathbb R)$. Specifically, Theorem 2.1 in \cite{BGM99} establishes that $\sqrt{n}(F_n-F)$ converges in distribution in $L^1(\mathbb R)$ if and only if
\begin{equation*}
	\int_{-\infty}^\infty\sqrt{F(x)(1-F(x))}\,\mathrm{d}x<\infty.
\end{equation*}
The latter condition is, by Lebesgue--Stieltjes substitution, equivalent to \eqref{eq:Qmoment} and therefore strictly weaker than Property \ref{prop:Q}. Our demonstration that Property \ref{prop:Q} is necessary and sufficient for $\sqrt{n}(Q_n-Q)$ to converge in distribution in $L^1(0,1)$ makes it easy to construct examples in which $\sqrt{n}(F_n-F)$ converges in distribution in $L^1(\mathbb R)$ but $\sqrt{n}(Q_n-Q)$ does not converge in distribution in $L^1(0,1)$. Any $Q$ that is not locally absolutely continuous but satisfies \eqref{eq:Qmoment} will do the job. A very simple example is obtained by choosing $Q$ to be the quantile function for the Bernoulli distribution. An example in which $Q$ is continuous but not locally absolutely continuous is obtained by choosing $Q$ to be the middle-thirds Cantor function or ``Devil's staircase''; see e.g.\ \citet[p.~170]{T11}. Either choice of $Q$ is sufficiently irregular to ensure that $\sqrt{n}(Q_n-Q)$ does not converge in distribution in $L^1(0,1)$. This may be surprising since the sequence of real-valued random variables $\sqrt{n}\smallint_0^1\lvert Q_n(u)-Q(u)\rvert\,\mathrm{d}u$ does converge in distribution for both of these choices of $Q$ and, more generally, for any choice of $Q$ satisfying \eqref{eq:Qmoment}. The latter convergence in distribution follows from the identity $\smallint_0^1\lvert Q_n(u)-Q(u)\rvert\,\mathrm{d}u=\smallint_{-\infty}^\infty\lvert F_n(x)-F(x)\rvert\,\mathrm{d}x$.

It may be illuminating to reflect on why convergence in distribution fails when sampling from the Bernoulli distribution with success probability $p\in(0,1)$. Let $p_n$ be the fraction of successes in a sample of size $n$. Then we have $Q(u)=\mathbbm 1(u>1-p)$ and $Q_n(u)=\mathbbm 1(u>1-p_n)$, and a simple argument using the central limit theorem shows that
\begin{equation*}
	\int_{1-p-\epsilon}^{1-p+\epsilon}\sqrt{n}(Q_n(u)-Q(u))\,\mathrm{d}u\rightsquigarrow N(0,p(1-p))\quad\text{for every small }\epsilon>0.
\end{equation*}
There exists no $L^1(0,1)$-valued random variable $\mathscr Q$ that satisfies
\begin{equation*}
	\int_{1-p-\epsilon}^{1-p+\epsilon}\mathscr Q(u)\,\mathrm{d}u\sim N(0,p(1-p))\quad\text{for every small }\epsilon>0
\end{equation*}
because the last integral must converge to zero as $\epsilon$ decreases to zero. Thus $\sqrt{n}(Q_n-Q)$ cannot converge in distribution to any random variable in $L^1(0,1)$. It is the discontinuity of $Q$ at $1-p$ that leads to the failure of convergence in distribution in this example. Property \ref{prop:Q} does not require $F$ to be continuous, but does require $F$ not to have flat segments within its convex support, as the height of each such segment is a discontinuity point of $Q$. In cases where $Q$ has a discontinuity point, the Lebesgue--Stieltjes measure generated by $Q$ assigns positive mass to that point, leading to irregular behavior of the quantile process. In cases where $Q$ is continuous but not locally absolutely continuous, the corresponding Lebesgue--Stieltjes measure is atomless but nevertheless assigns positive mass to a Lebesgue null set, leading to similar irregularity.

Property \ref{prop:Q} is strictly weaker than sufficient conditions for convergence in distribution of the quantile process in $L^1(0,1)$ appearing in prior literature. Proposition 1.4 in \cite{K18} establishes that the quantile process converges in distribution in $L^1(0,1)$ if (i) $F$ has at most finitely many discontinuities and is elsewhere continuously differentiable with positive derivative, and (ii) $\lvert Q\rvert^{2+\delta}$ is integrable for some $\delta>0$. Lemma A.18 in the online appendix to \cite{WPTM24} establishes that the quantile process converges in distribution in $L^1(0,1)$ if (a) $F$ is supported on a compact interval $[a,b]$ and (b) $F$ is continuously differentiable on $[a,b]$ with positive derivative on $(a,b)$. Both sets of conditions imply, but are not implied by, Property \ref{prop:Q}.

To prove the sufficiency of Property \ref{prop:Q} we take as a starting point the convergence in distribution of $\sqrt{n}(F_n-F)$ in $L^1(\mathbb R)$ supplied by Theorem 2.1(a) in \cite{BGM99}. From this we deduce the convergence in distribution of $\sqrt{n}(Q_n-Q)$ in $L^1(0,1)$ by applying the delta-method. The validity of a bootstrap approximation to the quantile process is established as a byproduct of the delta-method. We use a separate argument to prove the necessity of Property \ref{prop:Q}, drawing on Theorem 2.1(b) in \cite{BGM99}, which establishes that $\sqrt{n}\smallint_{-\infty}^\infty\lvert F_n(x)-F(x)\rvert\,\mathrm{d}x$ is stochastically bounded if and only if \eqref{eq:Qmoment} is satisfied.

A large portion of the technical work done in this article is devoted to verifying the requisite Hadamard differentiability for our application of the delta-method. We use the standard definition of Hadamard differentiability stated in \citet[p.~296--7]{V98}. The idea of using the delta-method to study the behavior of the quantile process is, of course, not new, and can be traced back to \cite{V72} and \cite{R76}. What makes the application of the delta-method challenging in our setting is the generality of Property \ref{prop:Q}. Indeed, in our verification of Hadamard differentiability we assume only that $Q$ is locally absolutely continuous and integrable. The stronger integrability condition \eqref{eq:Qmoment} is needed to obtain the convergence in distribution of $\sqrt{n}(F_n-F)$ in $L^1(\mathbb R)$ but plays no role in verifying Hadamard differentiability of the mapping to quantile functions.

The contribution of this article is closely related to prior research on the approximation of a special construction of the uniform quantile process under weighted $L^p$-norms. In particular, the following fact is established in \cite{CHS93}. Let $U_{1,n}\leq U_{2,n}\leq\cdots\leq U_{n,n}$ be the order statistics of $n$ iid draws from the uniform distribution on $(0,1)$, and let $U_n:(0,1)\to\mathbb R$ be defined by
\begin{equation*}\label{eq:ueqf}
	U_n(u)=U_{k,n},\quad k/(n+2)<u\leq(k+1)/(n+2),\quad k=0,\dots,n+1,
\end{equation*}
where $U_{0,n}=0$ and $U_{n+1,n}=1$. Let $p\in(0,\infty)$, and let $w:(0,1)\to(0,\infty)$ be Borel measurable and locally bounded. Then
\begin{equation}\label{eq:wmoment}
	\int_0^1(u(1-u))^{p/2}w(u)\,\mathrm{d}u<\infty
\end{equation}
if and only if we can define a sequence of Brownian bridges $B_n$ such that
\begin{equation}\label{eq:strongapprox}
	\int_0^1\left|\sqrt{n}(u-U_n(u))-B_n(u)\right|^pw(u)\,\mathrm{d}u=o_\mathrm{P}(1).
\end{equation}
This result is contained in Theorem 1.2 in \cite{CHS93}; see also Theorem 1 in \citet[p.~470]{SW86}. Note that $U_n$ differs from the standard construction of an empirical quantile function, which is to take the generalized inverse of an empirical cdf, i.e.\ $Q_n(u)=\inf\{x:F_n(x)\geq u\}$. Instead, $U_n$ has been defined in such a way that $U_n(u)=0$ for all $u\in(0,1/(n+2))$ and $U_n(u)=1$ for all $u\in((n+1)/(n+2),1)$. If $U_n$ were defined to be an empirical quantile function in the standard way then it would be uniformly bounded away from zero and one almost surely for all $n$, and consequently the integral in \eqref{eq:strongapprox} would be almost surely infinite in cases where $w$ is not integrable. The special construction of $U_n$ is therefore critical to the equivalence of \eqref{eq:wmoment} and \eqref{eq:strongapprox}.

Minor variations upon the empirical quantile function $Q_n$ are commonly used in empirical research. For instance, in \citet{P79} it is suggested to use the continuous and piecewise linear estimator
\begin{equation*}
	\check{Q}_n(u)=(k-nu)X_{(k-1),n}+(nu-k+1)X_{k,n},\quad (k-1)/n\leq u\leq k/n,\quad k=1,\dots,n,
\end{equation*}
where $X_{1,n}\leq X_{2,n}\leq\cdots\leq X_{n,n}$ are the sample order statistics and $X_{0,n}\coloneqq X_{1,n}$. Assume that $Q\in L^1(0,1)$. A simple calculation shows that
\begin{equation}\label{eq:Parzen}
	\lVert\check{Q}_n-Q_n\rVert_1=\frac{X_{n,n}-X_{1,n}}{2n}\quad\text{and}\quad\int_0^1(\check{Q}_n(u)-Q(u))\,\mathrm{d}u=\bar{X}_n-\mathrm{E}(X_1)-\frac{X_{n,n}-X_{1,n}}{2n}.
\end{equation}
If $Q\in L^2(0,1)$ then $X_{1,n},X_{n,n}=o_\mathrm{P}(n^{1/2})$ and so the first equality in \eqref{eq:Parzen} implies that $\lVert\check{Q}_n-Q_n\rVert_1=o_\mathrm{p}(n^{-1/2})$. Thus our results establish that if $Q\in L^2(0,1)$ then Property \ref{prop:Q} is necessary and sufficient for $\sqrt{n}(\check{Q}_n-Q)$ to converge in distribution in $L^1(0,1)$. On the other hand, if $Q\notin L^2(0,1)$ then Property \ref{prop:Q} is not satisfied and it may be shown using the second equality in \eqref{eq:Parzen} that $\sqrt{n}\smallint_0^1(\check{Q}_n(u)-Q(u))\,\mathrm{d}u$ does not converge in distribution. Similar arguments apply to other common variants of $Q_n$ discussed in \citet{HF96}.

Our discussion of related research has barely scratched the surface of the enormous volume of relevant literature on the quantile process. We make no attempt at a survey, but point to \citet{C83} and \citet{CH93} for detailed treatments with a particular emphasis on strong approximations with respect to a uniform or weighted uniform norm, and to \cite{M84}, \cite{BGU05} and \cite{HKW21} for discussions focused on $L^2$-norms. We also point to \citet{ZH20} for a treatment based on the concepts of epi- and hypoconvergence introduced in \citet{BSV14}.

The remainder of this article is structured as follows. Section \ref{sec:diff} is concerned with the verification of Hadamard differentiability. Building on Section \ref{sec:diff}, we show in Section \ref{sec:Q} that Property \ref{prop:Q} is necessary and sufficient for convergence in distribution of the quantile process in $L^1(0,1)$. Appendix \ref{sec:geniv} presents three lemmas summarizing well-known properties of generalized inverses. We appeal to these lemmas at numerous points throughout Sections \ref{sec:diff} and \ref{sec:Q}.

\section{Hadamard differentiability of the generalized inverse}\label{sec:diff}

In this section we establish that a mapping from modified cdfs to quantile functions is Hadamard differentiable under conditions implied by Property \ref{prop:Q}. The precise statement is given in Theorem \ref{thm:inversediff}.

Let $D(\mathbb R)$ be the space of real c\`{a}dl\`{a}g functions on $\mathbb R$, and let $D^1(\mathbb R)$ be the subspace of integrable $h\in D(\mathbb R)$ equipped with the norm $\lVert h\rVert_1\coloneqq\smallint_{-\infty}^\infty\lvert h(x)\rvert\,\mathrm{d}x$. For each cdf $G\in D(\mathbb R)$, let $G^{-1}:(0,1)\to\mathbb R$ be the generalized inverse of $G$ defined by
\begin{equation*}\label{eq:geninv}
	G^{-1}(u)=\inf\{x\in\mathbb R:G(x)\geq u\}\quad\text{for all }u\in(0,1).
\end{equation*}
We call $G^{-1}$ the \emph{quantile function} for $G$. Each cdf $G\in D(\mathbb R)$ does not belong to $D^1(\mathbb R)$ because cdfs increase to one. However the modified cdf $\tilde{G}$ defined by
\begin{equation}\label{eq:modcdf}
	\tilde{G}(x)=G(x)-\mathbbm 1(x\geq0)\quad\text{for all }x\in\mathbb R
\end{equation}
belongs to $D^1(\mathbb R)$ if and only if $G^{-1}\in L^1(0,1)$. Note that \eqref{eq:modcdf} defines a one-to-one correspondence between cdfs and their modifications. We elaborate on the purpose of the modification to cdfs at the end of this section.

For each cdf $F\in D(\mathbb R)$ we define
\begin{gather}
	\alpha_F\coloneqq\inf\{x\in\mathbb R:F(x)>0\},\quad\beta_F\coloneqq\sup\{x\in\mathbb R:F(x)<1\},\label{eq:alphaFbetaF}\\
	\text{and}\quad\mathbb D_F\coloneqq\left\{\tilde{G}\in D^1(\mathbb R):G\text{ is a cdf, }G(\alpha_F-)=0\text{ and }G(\beta_F+)=1\right\}.\notag
\end{gather}
Note that $\alpha_F$ and $\beta_F$ need not be finite. If $\alpha_F=-\infty$ or if $\beta_F=+\infty$ then it should be understood that the respective condition $G(\alpha_F-)=0$ or $G(\beta_F+)=1$ is automatically satisfied. In the statement of Theorem \ref{thm:inversediff} we will fix a cdf $F$ such that $Q\coloneqq F^{-1}$ is integrable and locally absolutely continuous. Integrability implies that $\tilde{F}\in\mathbb D_F$. Local absolute continuity implies that $Q$ has a nonnegative and locally integrable density $q:(0,1)\to\mathbb R$, uniquely defined up to null sets. We will not assume that $Q$ satisfies \eqref{eq:Qmoment}.

Let $\phi:\mathbb D_F\to L^1(0,1)$ be the map defined for each $\tilde{G}\in\mathbb D_F$ by $\phi(\tilde{G})=G^{-1}$. Our objective is to verify that $\phi$ is Hadamard differentiable at $\tilde{F}$ with respect to a suitable tangent space. Let $C[0,1]$ be the space of real continuous functions on $[0,1]$. Our tangent space will be the subspace of $D^1(\mathbb R)$ defined by
\begin{equation*}
	T_F\coloneqq\left\{h\in D^1(\mathbb R):h=g\circ F\text{ for some }g\in C[0,1]\text{ such that }g(0)=g(1)=0\right\}.
\end{equation*}	

\begin{theorem}\label{thm:inversediff}
	Let $F$ be a cdf such that $Q\coloneqq F^{-1}$ is integrable and locally absolutely continuous, and let $q$ be a density for $Q$. Then $(h\circ Q)q\in L^1(0,1)$ for each $h\in D^1(\mathbb R)$. Moreover, $\phi:\mathbb D_F\subset D^1(\mathbb R)\to L^1(0,1)$ is Hadamard differentiable at $\tilde{F}$ tangentially to $T_F$, with Hadamard derivative $\dot{\phi}_{\tilde{F}}:T_F\to L^1(0,1)$ satisfying $\dot{\phi}_{\tilde{F}}(h)=-(h\circ Q)q$ for each $h\in T_F$. The last equality defines a continuous and linear extension of $\dot{\phi}_{\tilde{F}}$ to all of $D^1(\mathbb R)$.
\end{theorem}

Our proof of Theorem \ref{thm:inversediff} is lengthy due to the fact that we have placed no structure on the quantile function $Q$ beyond integrability and local absolute continuity. A much shorter proof could be provided if we were to assume, say, that $Q$ admits a density $q$ that is uniformly bounded away from zero and infinity. Such an assumption must be avoided because it is not implied by Property \ref{prop:Q}.

We begin with a series of three lemmas. The first lemma takes care of the first and last assertions of Theorem \ref{thm:inversediff}.

\begin{lemma}\label{lem:inversediff}
	Let $F$ be a cdf such that $Q\coloneqq F^{-1}$ is integrable and locally absolutely continuous, and let $q$ be a density for $Q$. Then $(h\circ Q)q\in L^1(0,1)$ for each $h\in D^1(\mathbb R)$, and $h\mapsto -(h\circ Q)q$ is continuous and linear as a map from $D^1(\mathbb R)$ into $L^1(0,1)$.
\end{lemma}
\begin{proof}
	By Lemma \ref{lemma:LS} we have
	\begin{equation*}
		\int_0^1\lvert h(Q(u))\rvert\,\mathrm{d}Q(u)=\int_{\alpha_F}^{\beta_F}\lvert h(Q(F(x)))\rvert\,\mathrm{d}x\quad\text{for each }h\in D^1(\mathbb R).
	\end{equation*}
	We may replace $\,\mathrm{d}Q(u)$ with $q(u)\,\mathrm{d}u$ in the first integral because $Q$ is locally absolutely continuous with density $q$. Local absolute continuity of $Q$ implies continuity of $Q$, and $0<F(x)<1$ for all $x\in(\alpha_F,\beta_F)$, so by Lemma \ref{lem:geninv}\ref{en:Finvcts} we may replace $Q(F(x))$ with $x$ in the second integral. Consequently
	\begin{equation}\label{eq:halphabeta}
		\int_{0}^{1}\lvert h(Q(u))\rvert q(u)\,\mathrm{d}u=\int_{\alpha_F}^{\beta_F}\lvert h(x)\rvert\,\mathrm{d}x\quad\text{for each }h\in D^1(\mathbb R).
	\end{equation}
	Thus $(h\circ Q)q\in L^1(0,1)$ for each $h\in D^1(\mathbb R)$. Linearity of $h\mapsto -(h\circ Q)q$ as a map from $D^1(\mathbb R)$ into $L^1(0,1)$ is obvious. Continuity of this map follows from the fact that $\lVert (h\circ Q)q\rVert_1\leq\lVert h\rVert_1$ for each $h\in D^1(\mathbb R)$, by \eqref{eq:halphabeta}.
\end{proof}

The second lemma provides an expression relating an additive perturbation to a cdf $F$ to the corresponding perturbation to the quantile function for $F$.
\begin{lemma}\label{lem:perturb}
	Let $F$ be a cdf such that $Q\coloneqq F^{-1}$ is continuous. Let $G$ be a cdf such that $G(\alpha_F-)=0$ and $G(\beta_F+)=1$. Let $h=G-F$. Then
	\begin{equation*}
		G^{-1}(u)-Q(u)=\int_0^1[\mathbbm{1}(v\geq u)-\mathbbm{1}(v+h(Q(v))\geq u)]\,\mathrm{d}Q(v)\quad\text{for each }u\in(0,1).
	\end{equation*}
\end{lemma}
\begin{proof}
	First we establish that
	\begin{equation}
		G^{-1}(u)-Q(u)=\int_{\alpha_F}^{\beta_F}[\mathbbm{1}(F(x)\geq u)-\mathbbm{1}(G(x)\geq u)]\,\mathrm{d}x\quad\text{for each }u\in(0,1).\label{eq:kaji}
	\end{equation}
	Observe that for each $u\in(0,1)$ we have
	\begin{align*}
		\int_{-\infty}^{Q(u)\vee G^{-1}(u)}\mathbbm 1(F(x)\geq u)\,\mathrm{d}x&=Q(u)\vee G^{-1}(u)-Q(u)\\
		\text{and}\quad\int_{-\infty}^{Q(u)\vee G^{-1}(u)}\mathbbm 1(G(x)\geq u)\,\mathrm{d}x&=Q(u)\vee G^{-1}(u)-G^{-1}(u).
	\end{align*}
	Taking the difference gives
	\begin{equation}
		\int_{-\infty}^{Q(u)\vee G^{-1}(u)}[\mathbbm 1(F(x)\geq u)-\mathbbm 1(G(x)\geq u)]\,\mathrm{d}x=G^{-1}(u)-Q(u).\label{eq:kaji2}
	\end{equation}
	Since $G(\alpha_F-)=0$ we have $\mathbbm 1(F(x)\geq u)=\mathbbm 1(G(x)\geq u)=0$ for all $x<\alpha_F$. Therefore we may replace the lower limit of integration in \eqref{eq:kaji2} with $\alpha_F$. We also have $\mathbbm 1(F(x)\geq u)=\mathbbm 1(G(x)\geq u)=1$ for all $x\geq Q(u)\vee G^{-1}(u)$ by Lemma \ref{lem:geninv}\ref{en:FFinviff}, and also for all $x\geq\beta_F$ because $G(\beta_F+)=1$. Therefore we may replace the upper limit of integration in \eqref{eq:kaji2} with $\beta_F$. This establishes \eqref{eq:kaji}.
	
	Since $Q$ is continuous, we have $Q(F(x))=x$ for all $x\in(\alpha_F,\beta_F)$ by Lemma \ref{lem:geninv}\ref{en:Finvcts}. Thus we may rewrite \eqref{eq:kaji} as
	\begin{align*}
		G^{-1}(u)-Q(u)=\int_{\alpha_F}^{\beta_F}[\mathbbm{1}(F(x)\geq u)-\mathbbm{1}(F(x)+h(Q(F(x)))\geq u)]\,\mathrm{d}x
	\end{align*}
	for each $u\in(0,1)$. Now the assertion to be proved follows from an application of Lemma \ref{lemma:LS}.
\end{proof}

The third lemma shows that the right-hand side of the equality asserted in Lemma \ref{lem:perturb} may be understood to define a Lipschitz continuous map from $D^1(\mathbb R)$ to $L^1(0,1)$.
\begin{lemma}\label{lem:Lipschitz}
	Let $F$ be a cdf such that $Q\coloneqq F^{-1}$ is continuous. Then
	\begin{equation*}
		\int_0^1\int_0^1\big|\mathbbm{1}(v+h_1(Q(v))\geq u)-\mathbbm{1}(v+h_2(Q(v))\geq u)\big|\,\mathrm{d}Q(v)\,\mathrm{d}u\leq\lVert h_1-h_2\rVert_1
	\end{equation*}
	for all $h_1,h_2\in D^1(\mathbb R)$.
\end{lemma}
\begin{proof}
	First observe that, for each $v\in(0,1)$,
	\begin{equation*}
		\int_0^1\big|\mathbbm{1}(v+h_1(Q(v))\geq u)-\mathbbm{1}(v+h_2(Q(v))\geq u)\big|\,\mathrm{d}u=\lvert h_1(Q(v))-h_2(Q(v))\rvert.
	\end{equation*}
	It therefore follows from Fubini's theorem that
	\begin{multline*}
		\int_0^1\int_0^1\big|\mathbbm{1}(v+h_1(Q(v))\geq u)-\mathbbm{1}(v+h_2(Q(v))\geq u)\big|\,\mathrm{d}Q(v)\,\mathrm{d}u\\=\int_0^1\big|h_1(Q(v))-h_2(Q(v))\big|\,\mathrm{d}Q(v).
	\end{multline*}
	Use Lemma \ref{lemma:LS} to obtain
	\begin{equation*}
		\int_0^1\big|h_1(Q(v))-h_2(Q(v))\big|\,\mathrm{d}Q(v)=\int_{\alpha_F}^{\beta_F}\big|h_1(Q(F(x)))-h_2(Q(F(x)))\big|\,\mathrm{d}x,
	\end{equation*}
	and then use Lemma \ref{lem:geninv}\ref{en:Finvcts} and the continuity of $Q$ to obtain
	\begin{equation*}
		\int_{\alpha_F}^{\beta_F}\big|h_1(Q(F(x)))-h_2(Q(F(x)))\big|\,\mathrm{d}x=\int_{\alpha_F}^{\beta_F}\big|h_1(x)-h_2(x)\big|\,\mathrm{d}x\leq\lVert h_1-h_2\rVert_1.\qedhere
	\end{equation*}
\end{proof}

\begin{proof}[Proof of Theorem \ref{thm:inversediff}]
	In view of Lemma \ref{lem:inversediff} it suffices to show that $\phi:\mathbb D_F\to L^1(0,1)$ is Hadamard differentiable at $\tilde{F}$ tangentially to $T_F$, with Hadamard derivative given by $\dot{\phi}_{\tilde{F}}(h)=-(h\circ Q)q$ for each $h\in T_F$. Let $(t_n)$ be a vanishing sequence of positive real numbers, and let $(h_n)$ be a sequence in $D^1(\mathbb R)$ satisfying $\tilde{F}+t_nh_n\in\mathbb D_F$ for each $n$ and $\lVert h_n-h\rVert_1\to0$ for some $h\in T_F$. Note that $F+t_nh_n$ is a cdf with integrable quantile function and that $\tilde{F}+t_nh_n$ is the modification of this cdf in the sense of \eqref{eq:modcdf}. Note also that $(h\circ Q)q\in L^1(0,1)$, by Lemma \ref{lem:inversediff}. Our task is to show that
	\begin{equation}
		\left\Vert t_n^{-1}[(F+t_nh_n)^{-1}-Q]+(h\circ Q)q\right\Vert_1\to0.\label{eq:qgoal}
	\end{equation}
	Define $\xi_n\coloneqq t_n^{-1}[(F+t_nh_n)^{-1}-Q]\in L^1(0,1)$. Lemma \ref{lem:perturb} establishes that
	\begin{equation*}\label{eq:substitution}
		\xi_n(u)=t_n^{-1}\int_0^1[\mathbbm{1}(v\geq u)-\mathbbm{1}(v+t_nh_n(Q(v))\geq u)]\,\mathrm{d}Q(v)\quad\text{for each }u\in(0,1).
	\end{equation*}
	Let $\check{g}$ be a function in $C[0,1]$ satisfying $\check{g}\circ F=h$ and $\check{g}(0)=\check{g}(1)=0$. Such a function exists because $h\in T_F$. Let $g$ be the restriction of $\check{g}$ to $(0,1)$. Define
	\begin{equation}\label{eq:zeta}
		\zeta_n(u)\coloneqq t_n^{-1}\int_0^1[\mathbbm{1}(v\geq u)-\mathbbm{1}(v+t_ng(v)\geq u)]\,\mathrm{d}Q(v)\quad\text{for each }u\in(0,1).
	\end{equation}
	Note that the integrand is zero outside of a compact subset of $(0,1)$ because $g(v)\to0$ as $v$ approaches either endpoint of $(0,1)$. Thus the integral is well-defined. Moreover, the integral is unaffected if $g(v)$ is replaced with $h(Q(v))$, by Lemma \ref{lem:FQ}. It therefore follows from Lemma \ref{lem:Lipschitz} that
	\begin{equation*}
		\int_0^1\lvert\xi_n(u)-\zeta_n(u)\rvert\,\mathrm{d}u\leq\lVert h_n-h\rVert_1.
	\end{equation*}
	Consequently $\zeta_n\in L^1(0,1)$. Moreover, since $\lVert h_n-h\rVert_1\to0$, we deduce that \eqref{eq:qgoal} is satisfied if and only if $\lVert\zeta_n+(h\circ Q)q\rVert_1\to0$. Since $\lVert (h\circ Q)q-gq\rVert_1=0$ by Lemma \ref{lem:FQ}, the proof is complete if we can show that $\lVert\zeta_n+gq\rVert_1\to0$.
	
	We divide our remaining arguments into four steps. In the first step we introduce auxiliary functions $\chi_n$, $\chi_{n,\epsilon}$ and $\zeta_{n,\epsilon}$, and establish relevant properties. In the second and third steps we use the auxiliary functions to separately bound the integral of $\lvert\zeta_n+gq\rvert$ over two complementary subsets of $(0,1)$. The subset dealt with in the second step contains points that are close to zero or one or at which $g$ is close to zero, and the subset dealt with in the third step contains points not close to zero or one at which $g$ is not close to zero. In the fourth step we bring everything together to show that $\left\Vert\zeta_n+gq\right\Vert_1\to0$.
	
	\smallskip		
	{\em Step 1. Introduce auxiliary functions and establish relevant properties.}
	
	Define the function $\chi_n:(0,1)^2\to\{0,1\}$ by
	\begin{equation*}
		\chi_n(u,v)\coloneqq\mathbbm 1\big(v+t_n[g(v)\wedge0]<u\leq v+t_n[g(v)\vee0]\big).
	\end{equation*}
	For each $u$ and $v$ in $(0,1)$ we have
	\begin{equation*}
		\mathrm{sgn}(g(v))\chi_n(u,v)=\mathbbm 1(v+t_ng(v)\geq u)-\mathbbm 1(v\geq u),
	\end{equation*}
	where $\mathrm{sgn}$ is the signum function. It therefore follows from \eqref{eq:zeta} that
	\begin{equation}
		\zeta_n(u)=-t_n^{-1}\int_0^1\chi_n(u,v)\mathrm{sgn}(g(v))\,\mathrm{d}Q(v).\label{eq:zetachi}
	\end{equation}
	For each $\epsilon>0$, define the function $\chi_{n,\epsilon}:(0,1)^2\to\mathbb R$ by
	\begin{equation*}
		\chi_{n,\epsilon}(u,v)\coloneqq\begin{dcases}\chi_{n}(u,v)\bigg/\int_0^1\chi_{n}(u,w)\,\mathrm{d}w&\text{if }\int_0^1\chi_{n}(u,w)\,\mathrm{d}w\geq t_n\sqrt{\epsilon}\\0&\text{otherwise,}\end{dcases}
	\end{equation*}
	and the function $\zeta_{n,\epsilon}:(0,1)\to\mathbb R$ by
	\begin{equation*}
		\zeta_{n,\epsilon}(u)\coloneqq-\int_0^1\chi_{n,\epsilon}(u,v)g(v)\,\mathrm{d}Q(v).
	\end{equation*}
	
	It will be useful to establish some inequalities involving the function $\chi_n$. First observe that, for each $v\in(0,1)$, $\chi_n(\cdot,v)$ is the indicator function of a set contained in an interval of length $t_n\lvert g(v)\vert$. Thus
	\begin{equation}
		\int_0^1\chi_n(u,v)\,\mathrm{d}u\leq t_n\lvert g(v)\rvert\quad\text{for all }n\in\mathbb N,\,v\in(0,1).\label{eq:chiduup}
	\end{equation}
	For each $\epsilon>0$, define
	\begin{align*}
		\delta_\epsilon&\coloneqq\sup\{\delta\in[0,1):\lvert g(v)-g(u)\rvert\leq\epsilon\text{ for all }u,v\in(0,1)\text{ such that }\lvert v-u\rvert\leq\delta\}\\
		\text{and }\quad N_{\epsilon,1}&\coloneqq\inf\{N\in\mathbb N:t_n\lvert g(u)\rvert<\delta_\epsilon\text{ for all }n\geq N,\,u\in(0,1)\}.		
	\end{align*}
	We have $\delta_\epsilon>0$ for every $\epsilon>0$ because $g$ is uniformly continuous, and thus $N_{\epsilon,1}<\infty$ for every $\epsilon>0$ because $g$ is uniformly bounded and $t_n\to0$. Observe that
	\begin{equation}
		\lvert v-u\rvert\leq t_n\lvert g(v)\rvert\quad\text{for all }n\in\mathbb N\text{ and }u,v\in(0,1)\text{ such that }\chi_n(u,v)\neq0.\label{eq:vu}
	\end{equation}
	Consequently,
	\begin{equation}
		\lvert g(v)-g(u)\rvert\leq\epsilon\quad\text{for all }\epsilon>0,\,n\geq N_{\epsilon,1}\text{, and }u,v\in(0,1)\text{ such that }\chi_n(u,v)\neq0.\label{eq:gepsilon}
	\end{equation}
	Let $\lambda$ denote the Lebesgue measure. We deduce from \eqref{eq:gepsilon} that for each $\epsilon>0$, each $n\geq N_{\epsilon,1}$, and each $u\in(0,1)$,
	\begin{align}
		\int_0^1\chi_{n}(u,v)\,\mathrm{d}v&=\int_0^1\mathbbm 1\big(\lvert g(v)-g(u)\rvert\leq\epsilon\big)\chi_{n}(u,v)\,\mathrm{d}v\notag\\
		&=\lambda\big\{v\in(0,1):v+t_n[g(v)\wedge0]<u\leq v+t_n[g(v)\vee0],\,\lvert g(v)-g(u)\rvert\leq\epsilon\big\}\notag\\
		&\leq\lambda\big\{v\in(0,1):v+t_n[(g(u)-\epsilon)\wedge0]<u\leq v+t_n[(g(u)+\epsilon)\vee0]\big\}\notag\\
		&=\lambda\big\{v\in(0,1):u-t_n[(g(u)+\epsilon)\vee0]\leq v< u-t_n[(g(u)-\epsilon)\wedge0]\big\}\notag\\
		&\leq\lambda\big\{v\in(0,1):u-t_n[g(u)\vee0]-t_n\epsilon\leq v< u-t_n[g(u)\wedge0]+t_n\epsilon\big\}.\label{eq:bounds1}
	\end{align}
	Thus
	\begin{equation}
		\int_0^1\chi_{n}(u,v)\,\mathrm{d}v\leq t_n\lvert g(u)\rvert+2t_n\epsilon\quad\text{for all }\epsilon>0,\,n\geq N_{\epsilon,1},\,u\in(0,1).\label{eq:chidvup}
	\end{equation}
	Similar to \eqref{eq:bounds1}, for each $\epsilon>0$, each $n\geq N_{\epsilon,1}$, and each $u\in(0,1)$, we have
	\begin{equation}
		\int_0^1\chi_{n}(u,v)\,\mathrm{d}v\geq\lambda\big\{v\in(0,1):u-t_n[g(u)\vee0]+t_n\epsilon\leq v< u-t_n[g(u)\wedge0]-t_n\epsilon\big\}.\label{eq:bounds2}
	\end{equation}
	For each $\epsilon>0$, define
	\begin{equation*}
		N_{\epsilon,2}\coloneqq\inf\{N\in\mathbb N:t_n\lvert g(u)\rvert<\epsilon\text{ for all }n\geq N,\,u\in(0,1)\}.
	\end{equation*}
	We have $N_{\epsilon,2}<\infty$ for each $\epsilon>0$ because $g$ is uniformly bounded and $t_n\to0$. For each $\epsilon\in(0,1/2)$, each $n\geq N_{\epsilon,2}$, and each $u\in(\epsilon,1-\epsilon)$ we have
	\begin{equation*}
		u-t_n[g(u)\vee0]+t_n\epsilon>0\quad\text{and}\quad u-t_n[g(u)\wedge0]-t_n\epsilon<1,
	\end{equation*}
	and thus, if also $n\geq N_{\epsilon,1}$, deduce from \eqref{eq:bounds2} that
	\begin{equation*}
		\int_0^1\chi_{n}(u,v)\,\mathrm{d}v\geq\lambda\big\{v\in\mathbb R:u-t_n[g(u)\vee0]+t_n\epsilon\leq v< u-t_n[g(u)\wedge0]-t_n\epsilon\big\}.
	\end{equation*}
	Consequently, we have
	\begin{equation}
		\int_0^1\chi_{n}(u,v)\,\mathrm{d}v\geq t_n\lvert g(u)\rvert-2t_n\epsilon\quad\text{for all }\epsilon\in(0,1/2),\,n\geq (N_{\epsilon,1}\vee N_{\epsilon,2}),\,u\in(\epsilon,1-\epsilon).\label{eq:chidvlo}
	\end{equation}
	Observe further that, in view of \eqref{eq:vu}, we have
	\begin{equation}
		\int_0^\epsilon\chi_n(u,v)\,\mathrm{d}u=0\quad\text{for all }\epsilon\in(0,1/2),\,n\geq N_{\epsilon,2},\,v\in(2\epsilon,1),\label{eq:chidueps1}
	\end{equation}
	and similarly
	\begin{equation}
		\int_{1-\epsilon}^1\chi_n(u,v)\,\mathrm{d}u=0\quad\text{for all }\epsilon\in(0,1/2),\,n\geq N_{\epsilon,2},\,v\in(0,1-2\epsilon).\label{eq:chidueps2}
	\end{equation}
	
	We will also need to suitably control integrals involving the function $\chi_{n,\epsilon}$. From \eqref{eq:chiduup} and the definition of $\chi_{n,\epsilon}$ we deduce that
	\begin{equation}
		\int_0^1\chi_{n,\epsilon}(u,v)\,\mathrm{d}u\leq\epsilon^{-1/2}\lvert g(v)\rvert\quad\text{for all }\epsilon>0,\,n\in\mathbb N,\,v\in(0,1).\label{eq:modchidu}
	\end{equation}
	For each $\epsilon\in(0,1/2)$ define
	\begin{equation}
		A_{\epsilon}\coloneqq\big\{u\in(\epsilon,1-\epsilon):\lvert g(u)\rvert>2\epsilon+\sqrt{\epsilon}\big\}\label{eq:Adef}
	\end{equation}
	and define $A_{\epsilon}^\mathrm{c}\coloneqq(0,1)\setminus A_\epsilon$. The bound \eqref{eq:chidvlo} implies that
	\begin{equation}
		\int_0^1\chi_{n}(u,v)\,\mathrm{d}v\geq t_n\sqrt{\epsilon}\quad\text{for all }\epsilon\in(0,1/2),\,n\geq (N_{\epsilon,1}\vee N_{\epsilon,2}),\,u\in A_\epsilon.\label{eq:gooddivisor}
	\end{equation}
	We therefore deduce from the definition of $\chi_{n,\epsilon}$ that
	\begin{equation}
		\int_0^1\chi_{n,\epsilon}(u,v)\,\mathrm{d}v=1\quad\text{for all }\epsilon\in(0,1/2),\,n\geq (N_{\epsilon,1}\vee N_{\epsilon,2}),\,u\in A_\epsilon.\label{eq:modchidv}
	\end{equation}
	
	\smallskip		
	{\em Step 2. Bound the integral of $\lvert\zeta_n+gq\rvert$ over $A_{\epsilon}^\mathrm{c}$.}
	
	Let $\epsilon$ be a point in $(0,1/2)$. We will separately bound the integrals of $\lvert\zeta_n\rvert$ and $\lvert gq\rvert$ over $A_{\epsilon}^\mathrm{c}$. For the latter integral we merely use \eqref{eq:Adef} to obtain the bound
	\begin{align}
		\int_{A_\epsilon^\mathrm{c}}\lvert g(u)q(u)\rvert\,\mathrm{d}u&\leq\int_0^1\mathbbm 1\big(\lvert g(u)\rvert\leq2\epsilon+\sqrt{\epsilon}\big)\lvert g(u)\rvert\,\mathrm{d}Q(u)+\int_0^\epsilon\lvert g(u)\rvert\,\mathrm{d}Q(u)\notag\\&\quad\quad+\int_{1-\epsilon}^1\lvert g(u)\rvert\,\mathrm{d}Q(u).\label{eq:Acbound5}
	\end{align}
	Next we bound the former integral. Let $n$ be an integer satisfying $n\geq (N_{\epsilon,1}\vee N_{\epsilon,2})$. Observe that
	\begin{flalign}
		\int_{A_\epsilon^\mathrm{c}}\lvert\zeta_n(u)\rvert\,\mathrm{d}u&=\int_{A_\epsilon^\mathrm{c}}\left|t_n^{-1}\int_0^1\chi_n(u,v)\mathrm{sgn}(g(v))\,\mathrm{d}Q(v)\right|\,\mathrm{d}u&&\text{by \eqref{eq:zetachi}}\notag\\
		&\leq t_n^{-1}\int_0^1\int_{A_\epsilon^\mathrm{c}}\chi_n(u,v)\,\mathrm{d}u\,\,\mathrm{d}Q(v)&&\text{by Fubini's theorem}\notag\\
		&\leq t_n^{-1}\int_0^1\int_0^1\mathbbm 1\big(\lvert g(u)\rvert\leq2\epsilon+\sqrt{\epsilon}\big)\chi_n(u,v)\,\mathrm{d}u\,\,\mathrm{d}Q(v)\notag\\
		&\quad\quad+\mathrlap{t_n^{-1}\int_0^1\int_0^\epsilon\chi_n(u,v)\,\mathrm{d}u\,\,\mathrm{d}Q(v)}&&\notag\\
		&\quad\quad+t_n^{-1}\int_0^1\int_{1-\epsilon}^1\chi_n(u,v)\,\mathrm{d}u\,\,\mathrm{d}Q(v)&&\text{by \eqref{eq:Adef}.}\label{eq:Acbound1}
	\end{flalign}
	We bound the three terms in the upper bound in \eqref{eq:Acbound1} in turn. For the first term, we have
	\begin{equation*}
		\int_0^1\mathbbm 1\big(\lvert g(u)\rvert\leq2\epsilon+\sqrt{\epsilon}\big)\chi_n(u,v)\,\mathrm{d}u\leq\int_0^1\mathbbm 1\big(\lvert g(v)\rvert\leq3\epsilon+\sqrt{\epsilon}\big)\chi_n(u,v)\,\mathrm{d}u
	\end{equation*}
	by \eqref{eq:gepsilon}, and thus
	\begin{equation}
		t_n^{-1}\int_0^1\int_0^1\mathbbm 1\big(\lvert g(u)\rvert\leq2\epsilon+\sqrt{\epsilon}\big)\chi_n(u,v)\,\mathrm{d}u\,\,\mathrm{d}Q(v)\leq\int_0^1\mathbbm 1\big(\lvert g(v)\rvert\leq3\epsilon+\sqrt{\epsilon}\big)\lvert g(v)\rvert\,\mathrm{d}Q(v)\label{eq:Acbound2}
	\end{equation}
	by \eqref{eq:chiduup}. For the second term, we have
	\begin{align}
		t_n^{-1}\int_0^1\int_0^\epsilon\chi_n(u,v)\,\mathrm{d}u\,\,\mathrm{d}Q(v)&= t_n^{-1}\int_0^{2\epsilon}\int_0^\epsilon\chi_n(u,v)\,\mathrm{d}u\,\,\mathrm{d}Q(v)&&\text{by \eqref{eq:chidueps1}}\notag\\
		&\leq\int_0^{2\epsilon}\lvert g(v)\rvert\,\mathrm{d}Q(v)&&\text{by \eqref{eq:chiduup}.}\label{eq:Acbound3}
	\end{align}
	For the third term, a similar argument using \eqref{eq:chidueps2} in place of \eqref{eq:chidueps1} shows that
	\begin{equation}
		t_n^{-1}\int_0^1\int_{1-\epsilon}^1\chi_n(u,v)\,\mathrm{d}u\,\,\mathrm{d}Q(v)\leq\int_{1-2\epsilon}^1\lvert g(v)\rvert\,\mathrm{d}Q(v).\label{eq:Acbound4}
	\end{equation}
	Combining \eqref{eq:Acbound1}--\eqref{eq:Acbound4}, we obtain
	\begin{align*}
		\int_{A_\epsilon^\mathrm{c}}\lvert\zeta_n(u)\rvert\,\mathrm{d}u&\leq\int_0^1\mathbbm 1\big(\lvert g(u)\rvert\leq3\epsilon+\sqrt{\epsilon}\big)\lvert g(u)\rvert\,\mathrm{d}Q(u)+\int_0^{2\epsilon}\lvert g(u)\rvert\,\mathrm{d}Q(u)\\
		&\quad\quad+\int_{1-2\epsilon}^1\lvert g(u)\rvert\,\mathrm{d}Q(u).
	\end{align*}
	Together with \eqref{eq:Acbound5}, the last inequality establishes that
	\begin{align}
		\int_{A_\epsilon^\mathrm{c}}\lvert\zeta_n(u)+g(u)q(u)\rvert\,\mathrm{d}u&\leq2\int_0^1\mathbbm 1\big(\lvert g(u)\rvert\leq3\epsilon+\sqrt{\epsilon}\big)\lvert g(u)\rvert\,\mathrm{d}Q(u)\notag\\
		&\quad\quad+2\int_0^{2\epsilon}\lvert g(u)\rvert\,\mathrm{d}Q(u)+2\int_{1-2\epsilon}^1\lvert g(u)\rvert\,\mathrm{d}Q(u).\label{eq:Ac}
	\end{align}
	Note that we have required $\epsilon\in(0,1/2)$ and $n\geq (N_{\epsilon,1}\vee N_{\epsilon,2})$.
	
	\smallskip		
	{\em Step 3. Bound the integral of $\lvert\zeta_n+gq\rvert$ over $A_{\epsilon}$.}
	
	Again let $\epsilon$ be a point in $(0,1/2)$ and let $n$ be an integer satisfying $n\geq (N_{\epsilon,1}\vee N_{\epsilon,2})$. We will separately bound the integrals of $\lvert\zeta_n-\zeta_{n,\epsilon}\rvert$ and $\lvert\zeta_{n,\epsilon}+gq\rvert$ over $A_\epsilon$. First we bound the former integral. From \eqref{eq:zetachi} and the definition of $\zeta_{n,\epsilon}$ we have
	\begin{equation*}
		\zeta_n(u)-\zeta_{n,\epsilon}(u)=-\int_0^1[t_n^{-1}\chi_n(u,v)-\chi_{n,\epsilon}(u,v)\lvert g(v)\rvert]\mathrm{sgn}(g(v))\,\mathrm{d}Q(v)
	\end{equation*}
	for all $u\in A_\epsilon$, and thus by Fubini's theorem
	\begin{equation}
		\int_{A_{\epsilon}}\lvert\zeta_n(u)-\zeta_{n,\epsilon}(u)\rvert\,\mathrm{d}u\leq\int_0^1\int_{A_{\epsilon}}\left|t_n^{-1}\chi_n(u,v)-\lvert g(v)\rvert\chi_{n,\epsilon}(u,v)\right|\,\mathrm{d}u\,\,\mathrm{d}Q(v).\label{eq:A1bound1}
	\end{equation}
	In view of \eqref{eq:gooddivisor} and the definition of $\chi_{n,\epsilon}$, the inner integral on the right-hand side satisfies
	\begin{equation*}
		\int_{A_{\epsilon}}\left|t_n^{-1}\chi_n(u,v)-\lvert g(v)\rvert\chi_{n,\epsilon}(u,v)\right|\,\mathrm{d}u=\int_{A_{\epsilon}}\chi_{n,\epsilon}(u,v)\left|t_n^{-1}\int_0^1\chi_n(u,w)\,\mathrm{d}w-\lvert g(v)\rvert\right|\,\mathrm{d}u.
	\end{equation*}
	Thus, by the triangle inequality,
	\begin{multline}
		\int_{A_{\epsilon}}\left|t_n^{-1}\chi_n(u,v)-\lvert g(v)\rvert\chi_{n,\epsilon}(u,v)\right|\,\mathrm{d}u\leq\int_{A_{\epsilon}}\chi_{n,\epsilon}(u,v)\bigl\vert\lvert g(u)\rvert-\lvert g(v)\rvert\bigr\vert\,\mathrm{d}u\\
		+\int_{A_\epsilon}\chi_{n,\epsilon}(u,v)\left|t_n^{-1}\int_0^1\chi_n(u,w)\,\mathrm{d}w-\lvert g(u)\rvert\right|\,\mathrm{d}u.\label{eq:A1bound2}
	\end{multline}
	The definition of $\chi_{n,\epsilon}$ implies that $\chi_{n,\epsilon}(u,v)=0$ whenever $\chi_n(u,v)=0$, so from \eqref{eq:gepsilon} we have
	\begin{equation}
		\int_{A_{\epsilon}}\chi_{n,\epsilon}(u,v)\bigl\vert\lvert g(u)\rvert-\lvert g(v)\rvert\bigr\vert\,\mathrm{d}u\leq\epsilon\int_{A_{\epsilon}}\chi_{n,\epsilon}(u,v)\,\mathrm{d}u.\label{eq:A1bound3}
	\end{equation}
	From \eqref{eq:chidvup} and \eqref{eq:chidvlo} we have
	\begin{equation*}
		\left|t_n^{-1}\int_0^1\chi_n(u,w)\,\mathrm{d}w-\lvert g(u)\rvert\right|\leq2\epsilon\quad\text{for all }u\in A_\epsilon,
	\end{equation*}
	and thus
	\begin{equation}
		\int_{A_\epsilon}\chi_{n,\epsilon}(u,v)\left|t_n^{-1}\int_0^1\chi_n(u,w)\,\mathrm{d}w-\lvert g(u)\rvert\right|\,\mathrm{d}u\leq2\epsilon\int_{A_{\epsilon}}\chi_{n,\epsilon}(u,v)\,\mathrm{d}u.\label{eq:A1bound4}
	\end{equation}
	Consequently,
	\begin{align}
		\int_{A_{\epsilon}}\left|t_n^{-1}\chi_n(u,v)-\lvert g(v)\rvert\chi_{n,\epsilon}(u,v)\right|\,\mathrm{d}u&\leq3\epsilon\int_{A_{\epsilon}}\chi_{n,\epsilon}(u,v)\,\mathrm{d}u&&\text{by \eqref{eq:A1bound2}--\eqref{eq:A1bound4}}\notag\\
		&\leq3\sqrt{\epsilon}\lvert g(v)\rvert&&\text{by \eqref{eq:modchidu}.}\label{eq:A1bound5}
	\end{align}
	Combining \eqref{eq:A1bound1} and \eqref{eq:A1bound5}, we obtain
	\begin{equation}
		\int_{A_{\epsilon}}\lvert\zeta_n(u)-\zeta_{n,\epsilon}(u)\rvert\,\mathrm{d}u\leq3\sqrt{\epsilon}\int_0^1\lvert g(u)\rvert\,\mathrm{d}Q(u).\label{eq:A1}
	\end{equation}
	
	Next we bound the integral of $\lvert\zeta_{n,\epsilon}+gq\rvert$ over $A_\epsilon$. From \eqref{eq:modchidv} and the definition of $\zeta_{n,\epsilon}$ we have
	\begin{equation*}
		\zeta_{n,\epsilon}(u)+g(u)q(u)=-\int_0^1\chi_{n,\epsilon}(u,v)[g(v)q(v)-g(u)q(u)]\,\mathrm{d}v
	\end{equation*}
	for all $u\in A_\epsilon$, and thus
	\begin{equation}
		\int_{A_{\epsilon}}\lvert\zeta_{n,\epsilon}(u)+g(u)q(u)\rvert\,\mathrm{d}u\leq\int_{A_{\epsilon}}\int_0^1\chi_{n,\epsilon}(u,v)\lvert g(v)q(v)-g(u)q(u)\rvert\,\mathrm{d}v\,\mathrm{d}u.\label{eq:A2bound1}
	\end{equation}
	Define the function $f:(0,1)\to\mathbb R$ by $f(u)\coloneqq g(u)q(u)$. Since $f\in L^1(0,1)$, there exists (see e.g.\ Theorem 1.3.20 in \citealp{T11}) a uniformly continuous function $f_\epsilon:(0,1)\to\mathbb R$ such that $\lVert f_\epsilon-f\rVert_1\leq\epsilon$. Using the triangle inequality and Fubini's theorem, we deduce from \eqref{eq:A2bound1} that
	\begin{align}
		\int_{A_{\epsilon}}\lvert\zeta_{n,\epsilon}(u)+g(u)q(u)\rvert\,\mathrm{d}u&\leq\int_0^1\int_{A_{\epsilon}}\chi_{n,\epsilon}(u,v)\lvert f_\epsilon(v)-f(v)\rvert\,\mathrm{d}u\,\mathrm{d}v\notag\\
		&\quad+\int_{A_{\epsilon}}\int_0^1\chi_{n,\epsilon}(u,v)\lvert f_\epsilon(u)-f(u)\rvert\,\mathrm{d}v\,\mathrm{d}u\notag\\
		&\quad+\int_{A_{\epsilon}}\int_0^1\chi_{n,\epsilon}(u,v)\lvert f_\epsilon(v)-f_\epsilon(u)\rvert\,\mathrm{d}v\,\mathrm{d}u.\notag
	\end{align}
	From \eqref{eq:modchidu} we have
	\begin{align*}
		\int_0^1\int_{A_{\epsilon}}\chi_{n,\epsilon}(u,v)\lvert f_\epsilon(v)-f(v)\rvert\,\mathrm{d}u\,\mathrm{d}v&\leq\epsilon^{-1/2}\int_0^1\lvert g(v)\rvert\lvert f_\epsilon(v)-f(v)\rvert\,\mathrm{d}v\leq\sqrt{\epsilon}\sup_{v\in(0,1)}\lvert g(v)\rvert,
	\end{align*}
	and from \eqref{eq:modchidv} we have
	\begin{equation*}
		\int_{A_{\epsilon}}\int_0^1\chi_{n,\epsilon}(u,v)\lvert f_\epsilon(u)-f(u)\rvert\,\mathrm{d}v\,\mathrm{d}u=\int_{A_c}\lvert f_\epsilon(v)-f(v)\rvert\,\mathrm{d}v\leq\epsilon.
	\end{equation*}
	Thus
	\begin{equation*}
		\int_{A_{\epsilon}}\lvert\zeta_{n,\epsilon}(u)+g(u)q(u)\rvert\,\mathrm{d}u\leq\sqrt{\epsilon}\sup_{u\in(0,1)}\lvert g(u)\rvert+\epsilon+\int_{A_{\epsilon}}\int_0^1\chi_{n,\epsilon}(u,v)\lvert f_\epsilon(v)-f_\epsilon(u)\rvert\,\mathrm{d}v\,\mathrm{d}u.
	\end{equation*}
	Together with \eqref{eq:A1}, the last inequality establishes that
	\begin{align}
		\int_{A_{\epsilon}}\lvert\zeta_{n}(u)+g(u)q(u)\rvert\,\mathrm{d}u&\leq3\sqrt{\epsilon}\int_0^1\lvert g(u)\rvert\,\mathrm{d}Q(u)+\sqrt{\epsilon}\sup_{u\in(0,1)}\lvert g(u)\rvert+\epsilon\notag\\&\quad+\int_{A_{\epsilon}}\int_0^1\chi_{n,\epsilon}(u,v)\lvert f_\epsilon(v)-f_\epsilon(u)\rvert\,\mathrm{d}v\,\mathrm{d}u.\label{eq:A2}
	\end{align}
	Note again that we have required $\epsilon\in(0,1/2)$ and $n\geq (N_{\epsilon,1}\vee N_{\epsilon,2})$.
	
	\smallskip		
	{\em Step 4. Show that $\lVert\zeta_n-gq\rVert_1\to0$.}
	
	The inequalities \eqref{eq:Ac} and \eqref{eq:A2} together establish the following fact: For every $\epsilon\in(0,1/2)$, there exists an integer $N_\epsilon$ and a uniformly continuous function $f_\epsilon:(0,1)\to\mathbb R$ such that, for every $n\geq N_{\epsilon}$,
	\begin{align*}
		\lVert\zeta_n+gq\rVert_1&\leq2\int_0^1\mathbbm 1\big(\lvert g(u)\rvert\leq3\epsilon+\sqrt{\epsilon}\big)\lvert g(u)\rvert\,\mathrm{d}Q(u)+2\int_0^{2\epsilon}\lvert g(u)\rvert\,\mathrm{d}Q(u)\notag\\
		&\quad+2\int_{1-2\epsilon}^1\lvert g(u)\rvert\,\mathrm{d}Q(u)+3\sqrt{\epsilon}\int_0^1\lvert g(u)\rvert\,\mathrm{d}Q(u)+\sqrt{\epsilon}\sup_{u\in(0,1)}\lvert g(u)\rvert+\epsilon\\
		&\quad+\int_{A_{\epsilon}}\int_0^1\chi_{n,\epsilon}(u,v)\lvert f_\epsilon(v)-f_\epsilon(u)\rvert\,\mathrm{d}v\,\mathrm{d}u.
	\end{align*}
	The first six terms on the right-hand side of this inequality do not depend on $n$. As mentioned above, we have $(h\circ Q)q\in L^1(0,1)$ by Lemma \ref{lem:inversediff} and $\lVert (h\circ Q)q-gq\rVert_1=0$ by Lemma \ref{lem:FQ}, so we know that $\smallint_0^1\lvert g(u)\rvert\,\mathrm{d}Q(u)<\infty$. Thus, by the dominated convergence theorem, the first four terms can be made arbitrarily small by choosing $\epsilon$ sufficiently small. Since $g$ is uniformly bounded, the fifth term can also be made arbitrarily small by choosing $\epsilon$ sufficiently small. The sixth term is simply $\epsilon$. If we can show that the seventh and final term converges to zero as $n\to\infty$ for every $\epsilon\in(0,1/2)$ then we are done.
	
	Fix arbitrary real numbers $\eta>0$ and $\epsilon\in(0,1/2)$. Since $f_\epsilon$ is uniformly continuous, there exists a positive real number $\kappa_{\epsilon,\eta}$ such that
	\begin{equation}
		\lvert f_\epsilon(v)-f_\epsilon(u)\rvert\leq\eta\quad\text{for all }u,v\in(0,1)\text{ such that }\lvert v-u\rvert\leq\kappa_{\epsilon,\eta}.\label{eq:fepsilonmod}
	\end{equation}
	Define
	\begin{equation*}
		M_{\epsilon,\eta}\coloneqq\inf\big\{N\in\mathbb N:t_n\lvert g(u)\rvert\leq\kappa_{\epsilon,\eta}\text{ for all }n\geq N,u\in(0,1)\big\}.
	\end{equation*}
	We have $M_{\epsilon,\eta}<\infty$ because $g$ is uniformly bounded and $t_n\to0$. We deduce from \eqref{eq:vu} that
	\begin{equation}
		\lvert v-u\rvert\leq\kappa_{\epsilon,\eta}\quad\text{for all }n\geq M_{\epsilon,\eta}\text{ and }u,v\in(0,1)\text{ such that }\chi_n(u,v)\neq0.\label{eq:fepsilonmod2}
	\end{equation}
	The definition of $\chi_{n,\epsilon}$ implies that $\chi_{n,\epsilon}(u,v)=0$ whenever $\chi_n(u,v)=0$. Therefore, for all $n\geq M_{\epsilon,\eta}$,
	\begin{align*}
		\int_{A_{\epsilon}}\int_0^1\chi_{n,\epsilon}(u,v)\lvert f_\epsilon(v)-f_\epsilon(u)\rvert\,\mathrm{d}v\,\mathrm{d}u&\leq\eta\int_{A_{\epsilon}}\int_0^1\chi_{n,\epsilon}(u,v)\,\mathrm{d}v\,\mathrm{d}u&&\text{by \eqref{eq:fepsilonmod} and \eqref{eq:fepsilonmod2}}\\
		&=\eta\int_{A_\epsilon}1\,\mathrm{d}u\leq\eta&&\text{by \eqref{eq:modchidv}.}
	\end{align*}
	Since $\eta$ may be chosen arbitrarily small, we are done.
\end{proof}

We close this section with a comment on the modification to cdfs defined in \eqref{eq:modcdf}. The purpose of this modification is to allow us to define the generalized inverse map $\phi$ on a subset of the space $D^1(\mathbb R)$. This would not be possible if we took the domain of $\phi$ to be a set of cdfs because cdfs are not integrable. By defining $\phi$ on a subset of $D^1(\mathbb R)$ we ensure that the norm applied on the tangent space is also defined everywhere on the domain of $\phi$, as is required for Hadamard differentiability. The more general property of quasi-Hadamard differentiability, introduced in \cite{BZ10,BZ16}, does not impose this requirement and could be used to justify our intended application of the delta-method with $\phi$ defined directly on a set of cdfs. However, in view of the fact that the empirical process $\sqrt{n}(F_n-F)$ may always be re-written as $\sqrt{n}(\tilde{F}_n-\tilde{F})$, there seems to be no drawback to defining $\phi$ on a set of modified cdfs rather than on a set of cdfs.

\section{The quantile process}\label{sec:Q}

Following some preliminary discussion in Section \ref{sec:prelimQ}, we will establish in Section \ref{sec:sufficiencyQ} that Property \ref{prop:Q} is sufficient for convergence in distribution of the quantile process in $L^1(0,1)$ and justifies the use of the bootstrap. The necessity of Property \ref{prop:Q} is established in Section \ref{sec:necessityQ}.

\subsection{Preliminary discussion}\label{sec:prelimQ}

Let $(\Omega,\mathcal F,\mathbb P)$ be a probability space, and $(X_n)$ a sequence of random variables $X_n:\Omega\to\mathbb R$. Assume $(X_n)$ is iid with cdf $F$, and define $Q\coloneqq F^{-1}$. For each $n\in\mathbb N$ let $F_n:\Omega\to D(\mathbb R)$ be the empirical distribution function defined by
\begin{equation*}
	F_n(x)=\frac{1}{n}\sum_{i=1}^n\mathbbm 1(X_i\leq x)\quad\text{for all }x\in\mathbb R,
\end{equation*}
and define the empirical quantile function $Q_n\coloneqq F_n^{-1}$.

To define bootstrap counterparts to $F_n$ and $Q_n$ we introduce a second probability space $(\Omega',\mathcal F',\mathbb P')$ on which bootstrap weights will be defined, and define the product probability space
\begin{equation*}
	(\bar{\Omega},\bar{\mathcal F},\bar{\mathbb P})\coloneqq(\Omega\times\Omega',\mathcal F\otimes\mathcal F',\mathbb P\otimes\mathbb P').
\end{equation*}
For simplicity, and because resampling procedures are not the primary focus of this article, we confine attention to the Efron bootstrap. For each $n\in\mathbb N$, let $W_{n,1},\dots,W_{n,n}$ be random variables on $\Omega'$ which take values only in $\{0,1,\dots,n\}$ and whose sum is equal to $n$. Each $n$-tuple $(W_{n,1},\dots,W_{n,n})$ is assumed to have the multinomial distribution based on $n$ draws from the categories $1,\dots,n$, with equal probability assigned to each category. For each $n\in\mathbb N$ let $F_n^\ast:\bar{\Omega}\to D(\mathbb R)$ be defined by
\begin{equation*}
	F^\ast_n(x)=\frac{1}{n}\sum_{i=1}^nW_{n,i}\mathbbm 1(X_i\leq x)\quad\text{for all }x\in\mathbb R.
\end{equation*}
We regard the generalized inverse $Q_n^\ast\coloneqq F_n^{\ast-1}$ to be a map from $\bar{\Omega}$ into $L^1(0,1)$. The maps $F_n^\ast$ and $Q_n^\ast$ may be understood to be Efron bootstrap counterparts to $F_n$ and $Q_n$.

The maps $Q_n$ and $Q_n^\ast$ are $L^1(0,1)$-valued random variables. This follows from a standard argument using Fubini's theorem and the separability of the space $L^1(0,1)$, the details of which we omit. We refer to $\sqrt{n}(Q_n-Q)$ as the \emph{quantile process}. If $Q\in L^1(0,1)$, which must be the case if $Q$ has Property \ref{prop:Q}, then the quantile process is an $L^1(0,1)$-valued random variable. In Section \ref{sec:sufficiencyQ} we will establish that if $Q$ has Property \ref{prop:Q} then the quantile process converges in distribution to $-qB$ in $L^1(0,1)$, where $B$ is a Brownian bridge. The following lemma due to \cite{CHS93} shows that if $Q$ has Property \ref{prop:Q} then $qB$ is integrable with probability one.
\begin{lemma}\label{lem:CHS}
	Let $B$ be a Brownian bridge. Then $\smallint_0^1\lvert B(u)\rvert\,\mathrm{d}Q(u)<\infty$ with probability one if and only if $Q$ satisfies \eqref{eq:Qmoment}.
\end{lemma}
Lemma \ref{lem:CHS} is obtained by applying Theorem 2.1 in \cite{CHS93} with $\xi$ a Brownian bridge and $\mu$ the Lebesgue--Stieltjes measure generated by $Q$. Condition (2.14) of Theorem 2.1 is satisfied with $r=2$ and $C=2/\pi$ because the square of the first absolute moment of a normal random variable with mean zero and variance $\sigma^2$ is $2\sigma^2/\pi$.

\subsection{Sufficiency of Property \ref{prop:Q}}\label{sec:sufficiencyQ}

In this section we establish the following result. The definitions of convergence in distribution ($\rightsquigarrow$) and of convergence in distribution conditionally on the data are the standard ones stated in \citet[258--259, 333]{V98} for sequences of maps into a metric space.
\begin{theorem}\label{theorem:QCLT}
	Assume that $Q$ has Property \ref{prop:Q}, let $q$ be a density for $Q$, and let $B$ be a Brownian bridge. Then $\sqrt{n}(Q_n-Q)\rightsquigarrow-qB$ in $L^1(0,1)$ and
	\begin{equation*}
		\left\Vert\sqrt{n}(Q_n-Q)+q\cdot\sqrt{n}(F_n\circ Q-F\circ Q)\right\Vert_1\to0\quad\text{in probability.}
	\end{equation*}
	Moreover, $\sqrt{n}(Q_n^\ast-Q_n)\rightsquigarrow-qB$ in $L^1(0,1)$ conditionally on the data in probability.
\end{theorem}

We will prove Theorem \ref{theorem:QCLT} by applying the delta-method. Any application of the delta-method requires a probabilistic ingredient and an analytic ingredient. Our analytic ingredient is the Hadamard differentiability of the generalized inverse map supplied by Theorem \ref{thm:inversediff}. Our probabilistic ingredient is supplied by Lemma \ref{lem:FCLT1}. The part of Lemma \ref{lem:FCLT1} not concerning the bootstrap is essentially contained in Theorem 2.1(a) in \cite{BGM99}, while the part concerning the bootstrap is established in the online appendix to \cite{BCM24}.
\begin{lemma}\label{lem:FCLT1}
	Assume that $Q$ satisfies \eqref{eq:Qmoment}, and let $B$ be a Brownian bridge. Then $\sqrt{n}(F_n-F)\rightsquigarrow B\circ F$ in $D^1(\mathbb R)$, and $\sqrt{n}(F_n^\ast-F_n)\rightsquigarrow B\circ F$ in $D^1(\mathbb R)$ conditionally on the data almost surely.
\end{lemma}

\begin{proof}[Proof of Theorem \ref{theorem:QCLT}]
	
	Begin by writing $\sqrt{n}(Q_n-Q)=\sqrt{n}(\phi(\tilde{F}_n)-\phi(\tilde{F}))$, where $\phi$ is the generalized inverse map appearing in Theorem \ref{thm:inversediff} and where $\tilde{F}_n$ and $\tilde{F}$ are the modifications to $F_n$ and $F$ constructed as in \eqref{eq:modcdf}. Since $\tilde{F}_n-\tilde{F}=F_n-F$ we have $\sqrt{n}(\tilde{F}_n-\tilde{F})\rightsquigarrow B\circ F$ in $D^1(\mathbb R)$ by Lemma \ref{lem:FCLT1}. Now the first two assertions of Theorem \ref{theorem:QCLT} follow from an application of the delta-method---see Theorem 20.8 in \citet[p.~297]{V98}---with Theorem \ref{thm:inversediff} supplying the requisite Hadamard differentiability of $\phi$. Note that $B\circ F$ belongs to the tangent space $T_F$ because the sample paths of a Brownian bridge are continuous and tied to zero. Note also that $\dot{\phi}_{\tilde{F}}(B\circ F)=-qB$ because $\lVert(B\circ F\circ Q-B)q\rVert_1=0$ by Lemma \ref{lem:FQ}. A similar argument using the delta-method for the bootstrap---see Theorem 23.9 in \citet[p.~333]{V98}---establishes the final assertion of Theorem \ref{theorem:QCLT}.
\end{proof}

We remark in passing that the Hadamard differentiability of $\phi$ can be used in conjunction with the results in \cite{V91} to establish an optimality property of the empirical quantile function. It is shown in that article that the property of asymptotic efficiency is preserved under Hadamard differentiable transformations. The empirical cdf is asymptotically efficient when the set of permitted distributions is unrestricted, and remains so when that set is restricted to those distributions having Property \ref{prop:Q}, because doing so does not affect the corresponding tangent space of score functions. Consequently the empirical quantile function inherits the asymptotic efficiency of the empirical cdf under Property \ref{prop:Q}, as does any Hadamard differentiable transformation of the empirical quantile function.

\subsection{Necessity of Property \ref{prop:Q}}\label{sec:necessityQ}

In this section we establish the following result.
\begin{theorem}\label{theorem:QCLT2}
	Assume that $Q\in L^1(0,1)$, and that there exists an $L^1(0,1)$-valued random variable $\mathscr Q$ such that $\sqrt{n}(Q_n-Q)\rightsquigarrow\mathscr Q$ in $L^1(0,1)$. Then $Q$ has Property \ref{prop:Q}.
\end{theorem}
To prove Theorem \ref{theorem:QCLT2} we will apply the following lemma. Part \ref{en:QCLT2-2-1} is Theorem 2.1(b) in \cite{BGM99}. Part \ref{en:QCLT2-2-2} follows from Theorem 2.4(a) in \cite{BGM99}.
\begin{lemma}\label{lem:QCLT2-2}
	Assume that $Q\in L^1(0,1)$, and for each $n\in\mathbb N$ define
	\begin{equation*}
		\zeta_n\coloneqq\sqrt{n}\int_{-\infty}^\infty\lvert F_n(x)-F(x)\rvert\,\mathrm{d}x.
	\end{equation*}
	Let $B$ be a Brownian bridge. Then:
	\begin{enumerate}[label=\upshape(\roman*)]
		\item $(\zeta_n)$ is stochastically bounded if and only if $Q$ satisfies \eqref{eq:Qmoment}.\label{en:QCLT2-2-1}
		\item If $Q$ satisfies \eqref{eq:Qmoment} then
		\begin{equation*}
			\lim_{n\to\infty}\mathbb E\zeta_n^2=\mathbb E\left(\int_0^1\lvert B(u)\rvert\,\mathrm{d}Q(u)\right)^2<\infty,
		\end{equation*}
		which implies that $(\zeta_n)$ is uniformly integrable.\label{en:QCLT2-2-2}
	\end{enumerate}
\end{lemma}

We require one further lemma for our proof of Theorem \ref{theorem:QCLT2}.
\begin{lemma}\label{lem:QCLT2-3}
	Assume that $Q$ satisfies \eqref{eq:Qmoment}, let $B$ be a Brownian bridge, and let $a,b\in(0,1)$ be continuity points of $Q$ with $a<b$. Then
	\begin{equation*}
		\liminf_{n\to\infty}\mathbb E\sqrt{n}\int_a^b\lvert Q_n(u)-Q(u)\rvert\,\mathrm{d}u\geq\mathbb E\int_{[a,b)}\lvert B(u)\rvert\,\mathrm{d}Q(u).
	\end{equation*}
\end{lemma}
\begin{proof}
	For each $n\in\mathbb N$ we have
	\begin{align}
		\int_a^b\lvert Q_n(u)-Q(u)\rvert\,\mathrm{d}u&=\int_{-\infty}^\infty\lvert F_n(x)-F(x)\rvert\,\mathrm{d}x-\int_0^a\lvert Q_n(u)-Q(u)\rvert\,\mathrm{d}u\notag\\&\quad\quad-\int_b^1\lvert Q_n(u)-Q(u)\rvert\,\mathrm{d}u\notag\\
		&\geq\int_{-\infty}^\infty\lvert F_n(x)-F(x)\rvert\,\mathrm{d}x-\int_{-\infty}^{Q_n(a)\vee Q(a)}\lvert F_n(x)-F(x)\rvert\,\mathrm{d}x\notag\\
		&\quad\quad-\int_{Q_n(b)\wedge Q(b)}^\infty\lvert F_n(x)-F(x)\rvert\,\mathrm{d}x\notag\\
		&=\int_{Q(a)}^{Q(b)}\lvert F_n(x)-F(x)\rvert\,\mathrm{d}x-\int_{Q_n(a)\wedge Q(a)}^{Q_n(a)\vee Q(a)}\lvert F_n(x)-F(x)\rvert\,\mathrm{d}x\notag\\
		&\quad-\int_{Q_n(b)\wedge Q(b)}^{Q_n(b)\vee Q(b)}\lvert F_n(x)-F(x)\rvert\,\mathrm{d}x.\label{eq:symdiff}
	\end{align}
	By Lemma \ref{lem:FCLT1} and the continuous mapping theorem, and by applying Lemma \ref{lemma:LS},
	\begin{equation}
		\sqrt{n}\int_{Q(a)}^{Q(b)}\lvert F_n(x)-F(x)\rvert\,\mathrm{d}x\rightsquigarrow\int_{[a,b)}\lvert B(u)\rvert\,\mathrm{d}Q(u)\quad\text{in }\mathbb R.\label{eq:symdiffconvd}
	\end{equation}
	For each $n\in\mathbb N$ and each $c\in\{a,b\}$ we have
	\begin{equation*}
		\sqrt{n}\int_{Q_n(c)\wedge Q(c)}^{Q_n(c)\vee Q(c)}\lvert F_n(x)-F(x)\rvert\,\mathrm{d}x\leq\lvert Q_n(c)-Q(c)\rvert\cdot\sup_{x\in\mathbb R}\sqrt{n}\lvert F_n(x)-F(x)\rvert.
	\end{equation*}
	The sequence of random variables $\sup_{x\in\mathbb R}\sqrt{n}\lvert F_n(x)-F(x)\rvert$, $n\in\mathbb N$, converges in distribution in $\mathbb R$ by Donsker's theorem. Moreover, since $a$ and $b$ are continuity points of $Q$, Lemma 21.2 in \cite{V98} and the Glivenko--Cantelli theorem together imply that $Q_n(c)\to Q(c)$ with probability one. Thus, by the Slutsky theorem, for each $c\in\{a,b\}$ we have
	\begin{equation}
		\sqrt{n}\int_{Q_n(c)\wedge Q(c)}^{Q_n(c)\vee Q(c)}\lvert F_n(x)-F(x)\rvert\,\mathrm{d}x\rightsquigarrow0\quad\text{in }\mathbb R.\label{eq:symdiffop1}
	\end{equation}
	In view of the uniform integrability condition established in Lemma \ref{lem:QCLT2-2}\ref{en:QCLT2-2-2}, the statements of convergence in distribution in \eqref{eq:symdiffconvd} and \eqref{eq:symdiffop1} imply statements of convergence in mean:
	\begin{gather}
		\mathbb E\sqrt{n}\int_{Q(a)}^{Q(b)}\lvert F_n(x)-F(x)\rvert\,\mathrm{d}x\to\mathbb E\int_{[a,b)}\lvert B(u)\rvert\,\mathrm{d}Q(u)\label{eq:convmean1}\\
		\text{and}\quad\mathbb E\sqrt{n}\int_{Q_n(c)\wedge Q(c)}^{Q_n(c)\vee Q(c)}\lvert F_n(x)-F(x)\rvert\,\mathrm{d}x\to0\quad\text{for each }c\in\{a,b\}.\label{eq:convmean2}
	\end{gather}
	By combining \eqref{eq:symdiff}, \eqref{eq:convmean1}, and \eqref{eq:convmean2}, we obtain the assertion of Lemma \ref{lem:QCLT2-3}.
\end{proof}

\begin{proof}[Proof of Theorem \ref{theorem:QCLT2}]
	The assumed convergence in distribution in $L^1(0,1)$ implies, via the continuous mapping theorem, that the sequence $\sqrt{n}\smallint_0^1\lvert Q_n(u)-Q(u)\rvert\,\mathrm{d}u$, $n\in\mathbb N$, is stochastically bounded. Since $\smallint_0^1\lvert Q_n(u)-Q(u)\rvert\,\mathrm{d}u=\smallint_{-\infty}^\infty\lvert F_n(x)-F(x)\rvert\,\mathrm{d}x$, this is precisely the sequence $(\zeta_n)$ defined in Lemma \ref{lem:QCLT2-2}. Thus Lemma \ref{lem:QCLT2-2}\ref{en:QCLT2-2-1} establishes that $Q$ satisfies \eqref{eq:Qmoment}.
	
	Let $B$ be a Brownian bridge. Our assumption that $\sqrt{n}(Q_n-Q)\rightsquigarrow\mathscr Q$ in $L^1(0,1)$ implies, via the continuous mapping theorem, that for every $a,b\in[0,1]$ with $a<b$ we have
	\begin{equation*}
		\sqrt{n}\int_a^b\lvert Q_n(u)-Q(u)\rvert\,\mathrm{d}u\rightsquigarrow\int_a^b\lvert\mathscr Q(u)\rvert\,\mathrm{d}u\quad\text{in }\mathbb R.
	\end{equation*}
	Lemma \ref{lem:QCLT2-2}\ref{en:QCLT2-2-2} and the inequality $\smallint_a^b\lvert Q_n(u)-Q(u)\rvert\,\mathrm{d}u\leq\smallint_{-\infty}^\infty\lvert F_n(x)-F(x)\rvert\,\mathrm{d}x$ together imply that the sequence of random variables $\sqrt{n}\smallint_a^b\lvert Q_n(u)-Q(u)\rvert\,\mathrm{d}u$, $n\in\mathbb N$, is uniformly integrable. Thus
	\begin{equation}
		\mathbb E\sqrt{n}\int_a^b\lvert Q_n(u)-Q(u)\rvert\,\mathrm{d}u\to\mathbb E\int_a^b\lvert\mathscr Q(u)\rvert\,\mathrm{d}u<\infty\quad\text{for all }a,b\in[0,1],\,\,a<b.\label{eq:limE}
	\end{equation}
	
	We next establish that
	\begin{equation}
		\mathbb E\int_a^b\lvert\mathscr Q(u)\rvert\,\mathrm{d}u\geq\mathbb E\int_{[a,b]}\lvert B(u)\rvert\,\mathrm{d}Q(u)\quad\text{for all }a,b\in(0,1),\,\,a<b.\label{eq:openineq}
	\end{equation}
	Fix $a,b\in(0,1)$ with $a<b$. Since $Q$ is monotone, it has at most countably many discontinuities. Thus we may choose an increasing sequence $(a_k)$ of continuity points of $Q$ converging to $a$, and a decreasing sequence $(b_k)$ of continuity points of $Q$ converging to $b$. We have $\mathbb E\smallint_0^1\lvert\mathscr Q(u)\rvert\,\mathrm{d}u<\infty$ from \eqref{eq:limE}. This justifies using the dominated convergence theorem to write
	\begin{equation}
		\mathbb E\int_a^b\lvert\mathscr Q(u)\rvert\,\mathrm{d}u=\lim_{k\to\infty}\mathbb E\int_{a_k}^{b_k}\lvert\mathscr Q(u)\rvert\,\mathrm{d}u.\label{eq:closed1}
	\end{equation}
	We have $\mathbb E\smallint_0^1\lvert B(u)\rvert\,\mathrm{d}Q(u)<\infty$ by Lemma \ref{lem:QCLT2-2}\ref{en:QCLT2-2-2} because $Q$ satisfies \eqref{eq:Qmoment}. This justifies using the dominated convergence theorem to write
	\begin{equation}\mathbb E\int_{[a,b]}\lvert B(u)\rvert\,\mathrm{d}Q(u)=\lim_{k\to\infty}\mathbb E\int_{[a_k,b_k)}\lvert B(u)\rvert\,\mathrm{d}Q(u).\label{eq:closed2}
	\end{equation}
	Since each $a_k$ and $b_k$ is a continuity point of $Q$, Lemma \ref{lem:QCLT2-3} and \eqref{eq:limE} together imply that
	\begin{equation}
		\lim_{k\to\infty}\mathbb E\int_{a_k}^{b_k}\lvert\mathscr Q(u)\rvert\,\mathrm{d}u\geq\lim_{k\to\infty}\mathbb E\int_{[a_k,b_k)}\lvert B(u)\rvert\,\mathrm{d}Q(u).\label{eq:closed3}
	\end{equation}	
	Taken together, \eqref{eq:closed1}, \eqref{eq:closed2}, and \eqref{eq:closed3} establish that \eqref{eq:openineq} is satisfied.
	
	Finally we use \eqref{eq:openineq} to show that $Q$ is locally absolutely continuous. Let $\lambda$ be the Lebesgue measure on $(0,1)$, and let $\mu_Q$ be the Lebesgue--Stieltjes measure on $(0,1)$ generated by $Q$. Fix $\eta\in(0,1/2)$, and let $\mathcal I_\eta$ be the set of all finite unions of closed intervals with endpoints in $[\eta,1-\eta]$. It is enough to show that the restriction of $Q$ to $[\eta,1-\eta]$ is absolutely continuous. By definition, such absolute continuity is satisfied if for every $\epsilon>0$ there exists $\delta>0$ such that $\mu_Q(A)\leq\epsilon$ for every $A\in\mathcal I_\eta$ such that $\lambda(A)\leq\delta$.
	
	Fix $\epsilon>0$, and let $A$ be a set in $\mathcal I_\eta$. Since $A\subseteq[\eta,1-\eta]$, and since $u(1-u)\geq\eta(1-\eta)$ for each $u\in[\eta,1-\eta]$, we have
	\begin{equation}
		\mu_Q(A)=\int_{A}1\,\mathrm{d}Q(u)\leq\frac{1}{\sqrt{\eta(1-\eta)}}\int_{A}\sqrt{u(1-u)}\,\mathrm{d}Q(u).\label{eq:repeat1}
	\end{equation}
	The square of the first absolute moment of a normal random variable with mean zero and variance $\sigma^2$ is $2\sigma^2/\pi$. Thus, by Fubini's theorem,
	\begin{equation}
		\int_{A}\sqrt{u(1-u)}\,\mathrm{d}Q(u)=\sqrt{\frac{\pi}{2}}\mathbb E\int_{A}\lvert B(u)\rvert\,\mathrm{d}Q(u).\label{eq:repeat2}
	\end{equation}
	Since $A$ is a finite union of closed intervals, by combining \eqref{eq:openineq}, \eqref{eq:repeat1}, and \eqref{eq:repeat2} we obtain
	\begin{equation*}
		\mu_Q(A)\leq\sqrt{\frac{\pi}{2\eta(1-\eta)}}\mathbb E\int_{A}\lvert\mathscr Q(u)\rvert\,\mathrm{d}u.
	\end{equation*}
	Observe that, for each $N\in\mathbb N$,
	\begin{align*}
		\mathbb E\int_{A}\lvert\mathscr Q(u)\rvert\,\mathrm{d}u&=\mathbb E\int_{A}\mathbbm 1\big(\lvert\mathscr Q(u)\rvert\leq N\big)\lvert\mathscr Q(u)\rvert\,\mathrm{d}u+\mathbb E\int_{A}\mathbbm 1\big(\lvert\mathscr Q(u)\rvert> N\big)\lvert\mathscr Q(u)\rvert\,\mathrm{d}u\notag\\
		&\leq N\lambda(A)+\mathbb E\int_0^1\mathbbm 1\big(\lvert\mathscr Q(u)\rvert> N\big)\lvert\mathscr Q(u)\rvert\,\mathrm{d}u.\label{eq:startarg}
	\end{align*}
	By the dominated convergence theorem,
	\begin{equation*}
		\lim_{N\to\infty}\mathbb E\int_0^1\mathbbm 1\big(\lvert\mathscr Q(u)\rvert> N\big)\lvert\mathscr Q(u)\rvert\,\mathrm{d}u=0.
	\end{equation*}
	Thus there exists $N_0\in\mathbb N$, not depending on $A$, such that
	\begin{equation*}
		\mu_Q(A)\leq\sqrt{\frac{\pi}{2\eta(1-\eta)}}N_0\lambda(A)+\epsilon/2.
	\end{equation*}
	Now if we set
	\begin{equation*}
		\delta=\sqrt{\frac{\eta(1-\eta)}{2\pi N_0^2}}\,\epsilon
	\end{equation*}
	then $\mu_Q(A)\leq\epsilon$ if $\lambda(A)\leq\delta$. Thus the restriction of $Q$ to $[\eta,1-\eta]$ is absolutely continuous.
\end{proof}

\begin{appendix}
	
	\section{Generalized inverses}\label{sec:geniv}
	
	Here we state three lemmas used repeatedly in the arguments above. They describe well-known properties of generalized inverses. The first is contained in Proposition 1 in \citet{EH13}.
	\begin{lemma}\label{lem:geninv}
		Let $F:\mathbb R\to[0,1]$ be a cdf. Then:
		\begin{enumerate}[label=\upshape(\roman*)]
			\item If $x\in\mathbb R$ and $u\in(0,1)$ then $F(x)\geq u$ if and only if $x\geq F^{-1}(u)$.\label{en:FFinviff}
			\item If $u\in(0,1)$ and $u$ belongs to the range of $F$ then $F(F^{-1}(u))=u$.\label{en:FFinvF}
			\item If $F^{-1}$ is continuous then $F^{-1}(F(x))=x$ for all $x\in\mathbb R$ such that $0<F(x)<1$.\label{en:Finvcts}
		\end{enumerate}
	\end{lemma}
	
	The second lemma states a well-known substitution rule for Lebesgue--Stieltjes integrals.
	\begin{lemma}\label{lemma:LS}
		Let $F:\mathbb R\to[0,1]$ be a cdf, and let $Q=F^{-1}$. Let
		\begin{equation*}
			\alpha_F=\inf\{x\in\mathbb R:F(x)>0\}\quad\text{and}\quad\beta_F=\sup\{x\in\mathbb R:F(x)<1\}.
		\end{equation*}
		Let $h:(0,1)\to\mathbb R$ be a Borel measurable function that satisfies $\smallint_0^1\lvert h(u)\rvert\,\mathrm{d}Q(u)<\infty$ or is nonnegative. Then
		\begin{gather*}
			\int_{[a,b)} h(u)\,\mathrm{d}Q(u)=\int_{Q(a)}^{Q(b)} h(F(x))\,\mathrm{d}x\quad\text{for all }a,b\in(0,1)\text{ such that }a<b,\\
			\text{and}\quad\int_0^1h(u)\,\mathrm{d}Q(u)=\int_{\alpha_F}^{\beta_F}h(F(x))\,\mathrm{d}x.
		\end{gather*}
	\end{lemma}
	Lemma \ref{lemma:LS} may be proved using Theorem 3.6.1 in \cite{B07}. We omit the details. Lemmas \ref{lem:geninv} and \ref{lemma:LS} facilitate a short proof of our third and final lemma.
	\begin{lemma}\label{lem:FQ}
		Let $F:\mathbb R\to[0,1]$ be a cdf, let $Q=F^{-1}$, and let $\mu_Q$ be the Lebesgue--Stieltjes measure on $(0,1)$ generated by $Q$. Then $F(Q(u))=u$ for $\mu_Q$-a.e.\ $u\in(0,1)$.
	\end{lemma}
	\begin{proof}
		Use Lemma \ref{lemma:LS} to write $\smallint_0^1\lvert F(Q(u))-u\rvert\,\mathrm{d}Q(u)=\smallint_{\alpha_F}^{\beta_F}\lvert F(Q(F(x)))-F(x)\rvert\,\mathrm{d}x$ and then use Lemma \ref{lem:geninv}\ref{en:FFinvF} to replace $F(Q(F(x)))$ with $F(x)$ in the latter integral.
	\end{proof}
	
\end{appendix}

\end{document}